\documentclass[a4paper, 10pt, conference]{ieeeconf}      

\usepackage{pdfpages, amsfonts, amsmath, amssymb, color, graphicx, courier, listings, enumerate, bbold,pgf,tikz}
\usepackage{microtype, url, units, fancyhdr, fancybox, floatrow, booktabs, thmtools, subfigure,subfig, caption, mathrsfs, breqn,dsfont,accents}
\usepackage{parskip,verbatim}

\newtheorem{theorem}{Theorem}[section]
\newtheorem{lemma}[theorem]{Lemma}
\newtheorem{proposition}[theorem]{Proposition}
\newtheorem{corollary}[theorem]{Corollary}

\newenvironment{definition}[1][Definition.]{\begin{trivlist}
\item[\hskip \labelsep {\bfseries #1}]}{\end{trivlist}}

\newcommand{\x}[1]{\textrm{#1}}

\newcommand{\bv}{\begin{verbatim}}
\newcommand{\ev}{\end{verbatim}}
\newcommand*{\qed}{\hfill\ensuremath{\square}}%
\newcommand{\im}{\textrm{im}}

\newcommand{\bp}{\begin{pmatrix}}
\newcommand{\ep}{\end{pmatrix}}
\newcommand{\bi}{\begin{itemize}}
\newcommand{\ei}{\end{itemize}}
\newcommand{\bc}{\begin{cases}}
\newcommand{\ec}{\end{cases}}
\newcommand{\ba}{\[\begin{aligned}}
\newcommand{\ea}{\end{aligned}\]}
\newcommand{\be}{\begin{enumerate}}
\newcommand{\ee}{\end{enumerate}}
\newcommand{\eqn}[2]{\begin{equation}\label{#2} #1\end{equation}}
\newcommand{\eq}[1]{\begin{align*}#1\end{align*}}
\newcommand{\eqq}[2]{\begin{align}\label{#2} \begin{split} #1 \end{split} \end{align}}
\newcommand{\sbq}{\subseteq}
\newcommand{\spq}{\supseteq}
\newcommand{\half}{\frac{1}{2}}

\newcommand{\R}{\mathds{R}}
\newcommand{\C}{\mathds{C}}

\newcommand{\N}{\mathds{N}}

\newcommand{\one}{\mathds{1}}
\newcommand{\I}{\mathds{I}}
\newcommand{\V}{\vspace{3mm}}

\newcommand{\bs}{\backslash}

\newcommand{\issa}{\\ = &}
\newcommand{\isen}{=&}

\newcommand{\ub}[2]{\underbrace{#1}_{#2}}

\newcommand{\haakk}[1]{\left[ #1 \right]}

\newcommand{\lar}{\left( \begin{array}}
\newcommand{\rar}{\end{array} \right)}
\newcommand{\mc}[1]{\mathcal{#1}}

\newcommand{\md}[1]{\mathds{#1}}

\newcommand{\fone}[1]{\frac{1}{#1}}

\newlength{\dhatheight}
\newcommand{\doublehat}[1]{%
    \settoheight{\dhatheight}{\ensuremath{\hat{#1}}}%
    \addtolength{\dhatheight}{-0.35ex}%
    \hat{\vphantom{\rule{1pt}{\dhatheight}}%
    \smash{\hat{#1}}}}

\title{\LARGE \bf
Stability Analysis of Networked Systems Containing Damped and Undamped Nodes*
}

\author{Filip Koerts$^{1}$, Mathias B{\"u}rger$^{2}$, Arjan van der Schaft$^{3}$, 
Claudio De Persis$^{4}$
\thanks{*This work was not supported by any organization}
\thanks{$^{1}$ Filip Koerts is with Johann Bernoulli Institute for Mathematics and Computer Science, 	University of Groningen, The Netherlands
        {\tt\small f.j.koerts@rug.nl}}%
\thanks{$^{2}$ Mathias B{\"u}rger is researcher at Bosch, Stuttgart, Germany        
        {\tt\small mathias.buerger@ist.uni-stuttgart.de}}%
\thanks{$^{3}$ Arjan van der Schaft is with Johann Bernoulli Institute for Mathematics and Computer Science, 	University of Groningen, The Netherlands         
        {\tt\small a.j.van.der.schaft@rug.nl}}%
\thanks{$^{4}$ Claudio De Persis is with Engineering and Technology Institute, University of Groningen, The Netherlands
        {\tt\small c.de.persis@rug.nl}}%
}

\begin{document}

\maketitle
\thispagestyle{empty}
\pagestyle{empty}

\begin{abstract}
This paper answers the question if a qualitatively heterogeneous passive networked system containing damped and undamped nodes shows consensus in the output of the nodes in the long run. While a standard Lyapunov analysis shows that the damped nodes will always converge to a steady-state value, the convergence of the undamped nodes is much more delicate and depends on the parameter values of the network as well as on the topology of the graph. A complete stability analysis is presented based on an eigenvector analysis involving the mass values and the topology of both the original graph and the reduced graph obtained by a Kron reduction that eliminates the damped nodes. 
\end{abstract}

\small 


\normalsize

\section{Introduction}

One of the fundamental control problems related to network systems is the consensus or synchronization problem, where it is of interest to couple dynamical systems in such a way that they evolve asymptotically in an identical manner, see e.g. \cite{seyboth, grip}. Synchronization is a relevant stability-like property in numerous applications such as power systems, where frequencies of the power generators should be synchronized, which can be found in \cite{bergenhill}, or platooning vehicles, where the vehicles velocities should be synchronized, see e.g. \cite{olfati, traffic}.

Synchronization problems are particularly challenging if the individual systems are not identical but heterogeneous. There has been tremendous research on synchronization of heterogeneous systems (e.g., using dynamic coupling controllers) \cite{c2, c1, murguia}. A way do deal with heterogeneity in networks that is particularly relevant to this paper, is to exploit system properties such as passivity \cite{arcak, port2}. Roughly speaking, the conceptual idea is as follows. If all - possibly heterogeneous - nodes are strictly passive (w.r.t. the outputs used for coupling) and the couplings are passive then synchronization can be achieved. This conceptual idea is extremely powerful for the analysis of heterogeneous networks and extends in various directions.  E.g., passive networks are directly related to network optimization problems \cite{c1} and can exhibit complex dynamic behavior such as clustering \cite{c2}. Furthermore, passivity is also relevant in the analysis of power networks \cite{trip}. We study in this paper a basic class of passive networks, namely linear mass-spring-damper networks with constant external forces. While this model is simplified it captures many of the relevant properties of networks of passive systems as studied in \cite{c1,arcak,port2}. 

When considering synchronization, the relevant topological conditions on the network basically boil down to some form of connectivity notions. The relevance of the network topology becomes more evident if, e.g., the controllability of a network is analyzed (\cite{observability, rahmani}). When considering the whole network as one system, the controllability depends heavily on the topological location of the control inputs (i.e., the location of the controlled nodes) in the network. In \cite{rahmani}, the controllability of leader-follower consensus networks has been connected to the symmetry of the graph with respect to the control nodes (leaders). Similarly, the research direction of pinning control investigates the question, where to place a limited number of controllers in a network to achieve synchronization (see \cite{chen} for a survey).

In some sense, research on pinning control deals also with heterogeneous networks, while the heterogeneity is here qualitative (nodes are either controlled or uncontrolled). In fact, with such a qualitative heterogeneity in the network, the graph topology becomes highly important.

We study in this paper a class of passive networks with a qualitative heterogeneity. In particular, we consider mass-spring-type networks with many undamped and few damped nodes. This type of models might be applied to networks that contain a minority of nodes with damping constants being considerably higher than damping constants of other nodes. A natural approach would be to approximate damping values below some threshold value to zero. 

To study the convergence of the network, we use a Lyapunov function that is exactly the Lyapunov function used in \cite{c1}, \cite{arcak}, or \cite{port2}. However, as in this paper the network contains undamped nodes, convergence cannot be ensured with this Lyapunov function. In fact, depending on the location of the damped nodes, the network can exhibit oscillatory behavior. The main result of this paper is a set of fairly simple and easily verifiable graph theoretic conditions ensuring convergence. Our result has various  implications, which can be found in the concluding section. 

The remainder of the paper is organized as follows. The dynamical network model and the graph formalism is introduced in Section II.  System characteristics such as the network equilibrium and the ability to shift the equilibrium are covered in Section III. The convergence analysis is performed in Section IV, where first a Lyapunov analysis is presented, followed by a characterization of the invariant subspace, leading to the main result of the paper, a precise characterization of the convergence condition.


\section{Preliminaries}

We consider an undirected and connected graph $\mc G=(V,E)$ with $n$ nodes and $m$ edges and incidence matrix $B$. On each node $i \in V$, a dynamical system $\Sigma_i$ is placed which is modeled as
\eqq{ 
\dot p_i = -R_iy_i + u_i + v_i, \qquad y_i = M_i^{-1} p_i
}{nodes}
where $p_i \in \R^{r}$ and  $y_i \in \R^{r}$ are the state and the output, respectively. Further, we have the damping matrix $R_i \succcurlyeq 0$, inertia matrix $M_i \succ 0$, coupling input $u_i$ and a constant external input $v_i \in \R^r$. On each edge $k=(i,j) \in E$, a dynamical controller $\Gamma_k$ of dimension $r$ is placed with state $q_k \in \R^{r}$, output $f_k \in \R^{r}$ and is modeled as:
\eqq{ 
\dot q_k = z_k, \qquad f_k \isen W_kq_k
}{edges}
Here, $W_k \succ 0$ is the edge weight matrix of edge $k$. Variables without subscript denote the corresponding stacked variables of the plants and controllers. The coupling is established through 
\eqq{
u = -(B \otimes I_{r}) f, \qquad z = (B \otimes I_{r})^T y
}{coupling}
where $\otimes$ denotes the Kronecker product and $I_r$ the $r \times r$ identity matrix.  In the sequel, we will use the abbreviated notation $\mathds{B}:= B \otimes I_{r_c}$.  

\begin{definition} \textit{A node $i \in V$ is said to be damped if $R_i \succ 0$. A node is undamped if $R_i =0$, while it is partially undamped if $R_i$ is singular. }
\end{definition} 

{\bf Assumption.} The set $V$ of nodes is partitioned into a set $V_d$ of damped nodes with cardinality $n_d \geq 1$ and a set $V_u$ of (partially) undamped nodes with cardinality $n_u \geq 1$. Hence, there is at least one damped and at least one (partially) undamped node.  

{\bf Remark.}  
In the context of a mass-spring-damper system, node states $p_i$ can be seen as momenta of the masses, while $q_k$ is the elongation of the springs. The node output $y_i$ represents velocities, whereas the controller output $f_k$ stands for the force acting on its endpoints. 

\subsection{Closed-loop system}

Let $p=col(p_1, \dots, p_n)$, $q=col(q_1, \dots, q_m)$ be the stacked state vectors and similarly for the other variables. Taking \eqref{nodes}, \eqref{edges} and \eqref{coupling} together, we obtain the closed-loop system, denoted by $\Sigma \times \Gamma$ and whose state and output is denoted by $z:=col(p,q) \in \R^{r(n+m)}$ and $y \in \R^{rn}$, respectively. Its state-space representation reads as $\dot z=Az+Gv$, $y=M^{-1}p$, where

\eqq{
\ub{\bp \dot p \\ \dot q \ep}{\dot z} \isen \ub{\bp -RM^{-1} & -\md BW \\ \md B^TM^{-1} & 0 \ep}{A} \ub{\bp p \\ q \ep}{z} + \ub{\bp I \\ 0 \ep}{G} v \\
y \isen \bp M^{-1} & 0 \ep \bp p \\ q \ep 
}{y}

The system parameters are:
\bi 
\item $M:=diag(M_1, \dots, M_n) \succ 0$, a block diagonal matrix containing inertia matrices of the individual nodes.  
\item $R:=diag(R_1, \dots, R_n) \succcurlyeq 0$, a block diagonal matrix with damping matrices of the individual nodes. 
\item $W:=diag(W_1, \dots W_m) \succ 0$, the block diagonal matrix with spring constants (edge weights). 
\item $v$, a constant external input 
\ei

{\bf Remark.} The system $\Sigma \times \Gamma$ can be written in a port-Hamiltonian representation by using the Hamiltonian $H(p,q)=\half p^TMp + \half q^TWq$. This gives $\dot H=v^Ty-y^TRy \leq v^Ty$, which shows that $\Sigma \times \Gamma$ is passive, but not output strictly passive as $R$ is singular. Hence, this does not give us the wanted convergence results and we need to invoke LaSalle's Theorem (section IV).
\subsection{Second-order dynamics}

Since $\mc G$ is connected, $\x{rank}(B)=n-1$. Furthermore, $\ker(B^T)=\im(\one)$, where $\one$ is the stacked vector of all ones. Also, $\ker\mathds{B}=\im(\I_r)$, where $\I_r:= \one \otimes I_r$. In fact, $\md B$ represents a graph consisting of $r$ connected components that are copies of $\mc G$. 

A fundamental cycle matrix $C$ of $\mc G$ is a matrix of full column rank that satisfies $\ker(B)=\im(C)$, see e.g. \cite{graphmatrices}\footnote{The fundamental cycle matrix in \cite{graphmatrices} is the transposed of the fundamental cycle matrix used in this paper}. The full column rank matrix $\mathds{C}:= C \otimes I_r$ satisfies $\ker(\mathds{B})=\im(\mathds{C})$. 

Note that since $\dot q \in \im(\mathds{B}^T)$, the projection of $q$ onto $\im(\mathds{B}^T)^\perp=\ker(\mathds B)=\im(\mathds C)$ can be written as $\mathds C r$ for some $r \in \R^{r(m-n+1)}$ and is such that $q(t) \in \im(\mathds{B}^T) + \mathds{C} r$ for all $t \geq 0$. By integrating the output $y$, we obtain potentials or positions $s(t):= \int_0^t y(\tau) d\tau+s_0$, where $s_0$ satisfies $q(0)=\mathds{B}^Ts_0+W^{-1}\mathds{C} r$. This decomposition is possible and unique. All terms in the equation $\dot p=-Ry -\mathds{B}Wq +v$ can be written in terms of (derivatives) of $s$ and the result is a second order equation: 
\eqq{
M \ddot s = -R \dot s - \mathds{B}W\mathds{B}^T s + v
}{sor}
We define $\md L:= \md B W \md B^T$ to be the total Laplacian matrix. Notice that this is indeed a Laplacian matrix: it is symmetric and the row and column sums are zero. In fact, $\ker(\md L)=\ker(\md L^T)=\im(\I_r)$. Some off-diagonal entries of $\md L$ are positive if and only if there are $W_k$'s with negative off-diagional entries. This does not affect the stability, since $W$ is positive-semidefinite (see section IV). 

\subsection{Decomposition of $\md B$ and $\md L$}

The partitioning $\{V_d , V_u\}$ of $V$ also induces a partitioning of the edges into the set $E^d$ of edges between damped nodes, the set $E^u$ of edges between undamped nodes and the set $E^i$ of interconnecting edges between a damped and an undamped node. We obtain $\mc G=(V_d \cup V_u, E^d \cup E^i \cup E^u)$ with partitioned greater incidence matrix
\eq{
\md B= \bp \md B_d \\ \md B_u \ep =  \bp \md B_d^d & \md B^i_d & 0 \\ 0 & \md B^i_u & \md B^u_u \ep 
}
Let the edge weight matrix $W$ and the total Laplacian matrix $\md L$ be correspondingly decomposed. Now, decompose \eqref{sor} into blocks associated with the damped nodes, with subscript $d$, and (partially) undamped nodes, with subscript $u$. 
\eqq{
 \underbrace{\bp M_d & 0 \\ 0 & M_u \ep}_{M} \bp \ddot s_d \\ \ddot s_u \ep =& - \underbrace{\bp R_d & 0 \\ 0 & R_u \ep}_{R} \bp \dot s_d \\ \dot s_u \ep \\ & \hspace{-10mm} - \underbrace{\bp \md L^d_d+\md L_d^i & \md L^i_i \\ (\md L^i_i)^T & \md L^u_u+\md L_u^i \ep}_{L} \bp s_d \\ s_u \ep + \bp v_d \\ v_u \ep
}{el}
$\md L^d_d=\md B_d^dW^d(\md B_d^d)^T$ and $\md L^u_u=\md B_u^uW^u(\md B_u^u)^T$ are the Laplacian matrices corresponding to the subgraphs $\mc G_d:=(V_d,E^d)$ with incidence matrix $\md B_d^d$ and $\mc G_u:=(V_u,E^u)$ with incidence matrix $\md B_u^u$, respectively. The Laplacian matrix 
$$ \md L^i := \bp \md L_d^i & \md L_i^i \\ (\md L_i^i)^T & \md L_u^i \ep$$ 
corresponds to the subgraph $\mc G_i:=(V,E^i)$. $\md L_i^i=\md B_d^iW^i(\md B_u^i)^T$ contains the edge weight matrices of the interconnecting edges. Finally, $\md L_d^i=\md B_d^iW^i(\md B_d^i)^T$ and $\md L_u^i=\md B_u^iW^i(\md B_u^i)^T$ are positive semi-definite block diagonal matrices since in $\mc G_i$ there are no edges between two damped or two undamped nodes. 


In the sense of consensus dynamics, it is of interest to know whether the nodes show output consensus in the long run, that is, if $y$ converges to a point in $\im\{\I_r\}$ as $t \to \infty$. In case the system fails to show consensus, it is of interest to know the non-trivial steady-state behavior. From that, we derive useful information such as the degrees of freedom of the nodes at steady state. Therefore, we ask ourselves in this paper: 

\vspace{1mm}

{\bf Problem.} \textit{Does every plant output trajectory $y(t)$ of $\Sigma \times  \Gamma$, i.e. system \eqref{y}, converges to a point in the set $\im\{\I_r\}$? If not, what is the steady-state behavior of $\Sigma \times \Gamma$?}


\section{System characteristics} 

In this section, we give the equilibria and perform a shift so that the equilibrium is located at the origin. Note first that the affine subspaces of  $\mathds{R}^{rm} \backslash \im(\mathds{B}^T)$ are invariant under the dynamics of the controller state $q(t)$. We have the following characterization of $\mathds{B}$ and $\mathds{C}$: 

\begin{lemma}\label{bandc} $\im(\mathds{B}^T) \oplus  \im(W^{-1}\mathds{C}) = \mathds{R}^m$ and $\im(\mathds{B}^T) \cap  \im(W^{-1}\mathds{C}) = \{0\}$. 
\end{lemma}

\begin{proof} The first statement follows easily from $\im(\mathds{B}^T)^\perp = \ker(\mathds{B}) = \im(\mathds{C})$ and $W^{-1}$ is positive definite. The second statement holds since $\im(\mathds{B}^T) \cap \im(W^{-1} \mathds{C}) = \im(\mathds{B^T}) \cap \ker(\mathds{B}W)=\{0\}$. 
\end{proof}

As a result, the set of solutions $z(t)=col(p(t),q(t))$ of $\Sigma \times \Gamma$ where $q(t)$ is in one of the affine subspaces of $\mathds{R}^{rm} \backslash \im(\mathds{B}^T)$ is a shifted copy of those solutions of $\Sigma \times \Gamma$ where $q(t) \in \im(\md B^T)$. Consequently, without loss of generality we can assume that $q \in \im(\md B^T)$.

\begin{corollary}\label{copy} For every initial condition $z(0)=col(p(0), q(0))\in \mathds{R}^{r(n+m)}$, there exists a unique vector $\gamma \in \R^{r(m-n+1)}$ such that $q^*(0) := q(0) - W^{-1}\mathds{C} \gamma \in \im(\mathds{B}^T)$. Furthermore, with shifted initial conditions $z^*(0)=col(p(0), q^*(0))$, the trajectory difference $z^*(t) - z(t)$ is constant for all $t \geq 0$. 
\end{corollary}

\begin{proof} Existence and uniqueness of $\gamma$ follow from Lemma \ref{bandc}


Given that at some time $t \geq 0$, $p^*(t)=p(t)$ and $q^*(t)-q(t) \in \im(W^{-1}\mathds{C})$, we have $y^*(t)=y(t)$ and $\mathds{B}W(q^*(t)-q(t))=0$, hence $\dot p^*(t) -\dot p(t) = 0$. Also, $\dot q^*(t) - \dot q(t)= 0$.
\end{proof}

If we restrict $q$ to be in the invariant space of $\im(\md B^T)$, the system $\Sigma \times \Gamma$ has a unique equilibrium:

\begin{proposition}
The system $\Sigma \times \Gamma$ restricted to $q \in \im(\mathds{B}^T)$ has a unique equilibrium $\bar z = col(\bar p, \bar q)$ satisfying $\bar p=M \I_r \beta$ and $\{ \bar q \} = \haakk{W^{-1}\md B^\dagger(-R \I_r \beta +v) + \im(W^{-1}\md C)} \cap \im(\md B^T)$.\footnote{Here, $\dagger$ denotes the Moore-Penrose pseudoinverse.} In these expressions, $\beta$ is given by
\eqq{ \beta= ( \mathds{I}_r^T R \mathds{I}_r)^{-1} \mathds{I}_r^Tv
}{beta}
\end{proposition}


\begin{proof} From $\dot{\bar q}=0$, we obtain $\mathds{B}^TM^{-1}\bar p=0$. Since for connected graphs, $\ker(\mathds{B}^T)=\im \{ \I_r \}$, it follows that $\bar p \in \im \{ M \I_r \}$. Write $\bar p= M \I_r \beta$ for some $\beta \in \mathbb{R}^r$. The value of $\beta$ can be obtained by setting $\dot{\bar p}=0$, which gives $-R \I_r \beta + v \in \im(\md B)$ and consequently $-\I_r^TR \I_r \beta+ \I_r^T v \in \im(\I_r^T\md B)=\{0\}$. Since $R_i \succ 0$ for at least one $i \in V$, it follows that $\I_r^TR \I_r \succ 0$ and thus $\beta$ can be given uniquely as in \eqref{beta}. Substituting this result in the dynamics of $p$, we obtain $\md BW \bar q = -R \I_r \beta +v$, which gives $\bar q \in W^{-1}\md B^\dagger (-R \I_r \beta + v) + \im(W^{-1} \md C)$, with $\im(W^{-1} \md C) = \ker( \md B W)$. By assumption and by the dynamics of $q$, $\bar q \in \im(\md B^T)$. The intersection of both sets is a singleton by Lemma \ref{bandc}. 
\end{proof}

{\bf Remark.} The unique equilibrium point for $q(0) \in \im(\md B^T)$ corresponds to a state of output consensus since  $\bar v = M^{-1}\bar p \in \im\{ \I_r \}$. Also for $q(0) \notin \im(\md B^T)$, it is shown readily that $\bar v = \I_r \beta$ with $\beta$ being given by \eqref{beta}.

\subsection{Shifted model}

Now, we introduce shifted state variables so that the equilibrium coincides with the origin. The main benefit of doing this is that it allows to use common techniques to show output consensus in section 4.2. Besides that, we get rid of the constant input $v$ in the dynamics. Define $\tilde p(t)= p(t) - \bar p$, $\tilde q(t) =q(t) - \bar q$. Stack these together in the state vector $\tilde z=col(\tilde p, \; \tilde q)$ and define the output $\tilde y=M^{-1}p$, then we obtain the linear time-invariant (LTI) closed-loop system
\eqq{
\ub{\bp \dot{\tilde p} \\ \dot{\tilde q} \ep}{\dot{\tilde z}} \isen \ub{\bp -RM^{-1} & -\md BW \\ \md B^TM^{-1} & 0 \ep}{A} \ub{\bp \tilde p \\ \tilde q \ep}{\tilde z}  \\
\tilde y \isen \bp M^{-1} & 0 \ep \bp \tilde p \\ \tilde q \ep 
}{y2}

Since $\bar q \in \im(\md B^T)$, it follows that $q(t) \in \im(\md B^T)$ if and only if $\tilde q(t) \in \im(\md B^T)$. By assumption, $\Sigma \times \Gamma$ is only defined on the invariant subspace $$\Omega = \{ col(\tilde p, \; \tilde q) \in \R^{r(n+m)} \mid  \tilde q \in \im(\md B^T) \}$$
System \eqref{y2} defined on $\Omega$ has a unique equilibrium point at $\bar{\tilde z}=0$. Similarly to the procedure in section IIA, we can introduce variables $\tilde s(t) = col(\tilde s_d(t), \tilde s_u(t)) := \int \tilde y(\tau) d\tau + \tilde s_0$, where $\tilde s_0$ is such that $\tilde q(0)=\md B^T \tilde s_0$, to obtain the second-order equation $$M\ddot{\tilde s}=-R\dot{\tilde s}- \md L{\tilde s}$$ 
Noting that $\tilde q(t)=\md B^T \tilde s(t)$ and $\tilde p(t)=M \dot{\tilde s}(t)$, \eqref{y2} can be written equivalently as an LTI system with states $\tilde s$ and $\tilde y$. In the next section we find Lemma \ref{qook} that connects global asymptotic stability of \eqref{y2} defined on $\Omega$ with output consensus of \eqref{y}.

\section{Steady-state behavior}

In this section, we determine the long-run behavior of \eqref{y2} defined on $\Omega$ by performing a common Lyapunov analysis. This allows us to derive the set of points to which all solutions converge. To find necessary conditions for the steady-state behavior, we use as Lyapunov function the Hamiltonian function that has a minimum at the equilibrium point $(\bar{\tilde p}, \bar{\tilde q})=(0,0)$:
$$U(\tilde p,\tilde q)=\frac{1}{2} \tilde p^TM^{-1}\tilde p+ \frac{1}{2} \tilde q^TW\tilde q $$

The time derivative of $U(\tilde p, \tilde q)$ now reads as
\begin{align*}
\dot U(\tilde p,\tilde q) &= \dot{\tilde p}^TM^{-1}\tilde p+ \dot{\tilde q}^TW\tilde q \\
&= (-RM^{-1} \tilde p- \md BW\tilde q)^T   M^{-1}{\tilde p} +  {\tilde p}^T M^{-1} \md BW\tilde q  \\
&= -\bp  {\tilde p}_d^T & {\tilde p}_u^T \ep M^{-1} \bp R_d \\ & R_u \ep M^{-1} \bp  {\tilde p}_d \\ {\tilde p}_u \ep \\
&=- {\tilde p}_d^T M_d^{-1}R_dM_d^{-1}  {\tilde p}_d - \tilde p_u^T M_u^{-1} R_u M_u^{-1} \tilde p_u 
\end{align*}

From the fact that $M$ and $W$ are positive definite, $U$ is a postive-definite function for $(\tilde p,\tilde q) \neq (0,0)$, while $\dot U$ is negative semi-definite, $U$ is a suitable Lyapunov function. 

\begin{lemma}\label{stable} The system $\dot{\tilde z}= A \tilde z$ as defined in \eqref{y2}, is stable. \end{lemma}

\begin{proof} Since $\dot U$ is everywhere nonpositive on $\R^{r(n+m)}$, we deduce that $\tilde z(t)^T \tilde M \tilde z(t) \leq \tilde z(0)^T \tilde M \tilde z(0)$ for all $t \geq 0$, where $\tilde M:=diag(M^{-1},W) \succ 0$. Noting that $\lambda_{min}(\tilde M)>0$, the min-max theorem yields $\| \tilde z(t) \|_2 \leq \sqrt{\lambda_{max}(\tilde M)/\lambda_{min}(\tilde M)} \| \tilde z(0)\|_2 $ for all $t \geq 0$.
\end{proof}

\begin{lemma}\label{qook}
Every trajectory $y(t)$ of $\Sigma \times \Gamma$, i.e. system \eqref{y}, converges to a point in the set $\im\{ \I_r \}$ if and only if every trajectory $\tilde z(t)=col( \tilde p(t),  \tilde q(t))$ of \eqref{y2} defined on $\Omega$ converges to the origin. 
\end{lemma}

\begin{proof}  In this proof, every convergence statement holds exclusively for $t \to \infty$.
$(\Rightarrow)$ Suppose that every output trajectory $y(t)$ of $\Sigma \times \Gamma$ converges to $\im\{ \I_r \}$, then this holds in particular for those trajectories generated with $q(0) \in \im(\md B^T)$. So $p(t) \to p^*$, where $p^* \in \im\{ M \I_r \}$. From Lemma \ref{stable}, we deduce that $\ddot{p}(t)$ is bounded, hence $\dot p(t)$ is uniformly continuous and we can apply Barbalat's lemma to conclude that $\dot p(t) \to 0$. That gives $-RM^{-1}p(t)-\md BW  q(t)+v \to 0$. So $\md BWq(t)$ converges too and consequently, $q(t) \to q^* + \ker(\md BW)$ for some $q^* \in \R^m$. Since $q(t) \in \im(\md B^T)$, we have that $q(t) \to \im(\md B^T) \cap \haakk{q^* + \ker(\md BW)}$, which is a singleton, so $q(t)$ converges too and consequently, the whole state $z(t)$ converges. Uniqueness of the equilibrium implies that $z(t) \to \bar z$ and hence $\tilde z(t) \to 0$.
$(\Leftarrow)$ For every trajectory $\tilde z(t) \to 0$ we have $z(t) \to \bar z$, yielding $y(t) \to \bar y=M^{-1}\bar p =\I_r \beta$, with $\beta$ as in \eqref{beta}.
\end{proof}

Lemma \ref{qook} shows that output consensus in the long run of the system $\Sigma \times \Gamma$ is equivalent to global asymptotic stability (GAS) of the system $\eqref{y2}$ defined on $\Omega$ and we will interchangeably use both terms.

Now we use LaSalle's invariance principle, which is a necessary condition for the long-run behavior: as $t$ goes to infinity, the trajectory converges to the largest invariant set in the set of states where $\dot U=0$. 

\begin{lemma}\label{pede} Every trajectory $\tilde z^*(t)$ of \eqref{y2} defined on $\Omega$ converges to $\mc S^{LS}$, which we define to be the set of initial conditions $\tilde z(0)$ such that for all $t \geq 0$, $\tilde p_d(t)=0$ and $\tilde p_u(t) \in \ker(R_u M_u^{-1})$, where the trajectory $\tilde z(t)=col(\tilde p(t), \tilde q(t))$ is the solution of \eqref{y2} with initial condition $\tilde z(0)$. \end{lemma}

\begin{proof} The set of points in the state space where $\dot U=0$ is, by positive definiteness of $M_d^{-1}R_dM_d^{-1}$ and by $\ker(M_u^{-1}R_uM_u^{-1})=\ker(R_uM_u^{-1})$, equal to the set of points in the state space where ${\tilde p}_d =0$ and $\tilde p_u \in \ker(R_uM_u^{-1})$.
 \end{proof}

\subsection{Behavior of the undamped nodes at steady state}

In this subsection, we give a precise characterization of $\mc S^{LS}$ as defined in Lemma \ref{pede} and work towards an LTI system that is observable if and only if output consensus of $\Sigma \times \Gamma$ is achieved. We introduce the following terminology: denote by $Obs(C,A)$ the observability matrix associated with the pair $(C,A)$. In the following Lemma, $\mc S^{LS}$ is written as a linear transformation of the unobservable subspace of a linear time-invariant system that gives the steady-state behavior of the undamped nodes in the $(\tilde y_u,  \tilde s_u)$ coordinates. 

\begin{lemma}\label{bestbelangrijk} $\mc S^{LS} = \hat Q \ker(Obs(\hat C, \hat A))$, where
$$\hat Q= \bp 0 & 0 \\ M_u & 0 \\ 0 & \md B_u^T- \md B_d^T(\md L_d^d+\md L_d^i)^{-1}\md L_i^i \ep ,$$ 
\eq{
\hat C =  \bp  \md L_i^i & 0 \\ R_u & 0 \ep, \qquad
\hat A = \bp 0 &  -M_u^{-1} \tilde{\md L}_u \\ I & 0 \ep, 
}
\eqq{
\tilde{\md L}_u = (\md L_u^u+ \md L_u^i) - (\md L_i^i)^T (\md L_d^d+\md L_d^i)^{-1} \md L_i^i
}{tlu}
Here, the block rows of $\hat Q$ are decomposed according to the decomposition of $\tilde z(t) = col( \tilde p_d, \tilde p_u, \tilde q )$ and the block columns of $\hat Q$ according to the decomposition of $\hat A$. 
\end{lemma}

\begin{proof}
$(\sbq)$ Take $\tilde z(0) \in \mc S^{LS}$ and consider the trajectory $\tilde z(t) = e^{At} \tilde z(0) \in \Omega$, which is a solution to \eqref{y2}. Decompose $\tilde z$ as $col(\tilde p_d, \tilde p_u, \tilde q)$, then we have for all $t \geq 0$: $\tilde p_d(t)=0$, $\tilde p_u(t) \in \ker(R_uM_u^{-1})$, $\tilde q(t) \in \im(B^T)$ and
 $$\bp \dot{ \tilde p}_d \\ \dot{\tilde p}_u \\ \dot{\tilde q} \ep = \bp -R_dM_d^{-1} & 0 & -\md B_dW \\ 0 & -R_uM_u^{-1} & -\md B_uW \\ \md B_d^T M_d^{-1} & \md B_u^T M_u^{-1} & 0 \ep \bp { \tilde p}_d \\ {\tilde p}_u \\ {\tilde q} \ep$$
From $\tilde p_d \equiv 0$, it also follows that $\dot{\tilde p}_d \equiv 0$. Also, $-R_uM_u^{-1}\tilde p_u \equiv 0$. Substituting these results in the dynamics, we obtain that for all $t \geq 0$, $-\md B_dW \tilde q(t) = 0$ and 
\eqq{
\bp \dot{\tilde p}_u \\ \dot{\tilde q} \ep = \bp 0 & - \md B_uW \\ \md B_u^TM_u^{-1} & 0 \ep \bp \tilde p_u \\ \tilde q \ep
}{tored}
Define the function $\tilde y_u(t):=M_u^{-1}\tilde p_u(t)$ and the auxiliary function $\tilde y_d(t) := M_d^{-1} \tilde p_d(t) \equiv 0$. It follows directly that
\eqq{
R_u \tilde y_u(t) \equiv 0
}{ru}
Since $\tilde q \in \im(\md B^T)$, there exist initial positions $\tilde s_d(0)$ and $\tilde s_u(0)$ satisfying 
\eqq{
\tilde q(0) = \md B_d^T \tilde s_d(0) + \md B_u^T \tilde s_u(0)
}{qq}
Now, define the functions $\tilde s_u(t) := \int_0^t \tilde y_u(\tau) d \tau + \tilde s_u(0)$ and $\tilde s_d(t) := \int_0^t \tilde y_d(\tau) d \tau + \tilde s_d(0) \equiv \tilde s_d(0)$. From $\dot{\tilde q}(t) = \md B_u^T \tilde y_u(t)$, we have $\tilde q(t) = \int_0^t \md B_u^T \tilde y_u(\tau) d \tau + \tilde q(0)$, which is easily shown to satisfy $\tilde q(t)=\md B_d^T \tilde s_d(t) + \md B_u^T \tilde s_u(t)$ for all $t \geq 0$. 

From $\md B_dW \tilde q(t) \equiv 0$, it follows that $-\md B_dW\md B_d^T \tilde s_d(t) = \md B_dW\md B_u^T \tilde s_u(t)$, which, after exploiting the decomposition of $\md L$ as in \eqref{el}, results in 
\eqq{
\tilde s_d(t) \equiv -(\md L_d^d+\md L_d^i)^{-1} \md L_i^i \tilde s_u(t)
}{sdd} Also, from $-\md B_dW \tilde q(t) \equiv 0$, it follows that $-\md B_dW \dot{\tilde q}(t) = -\md B_dW \md B_u^T M_u^{-1} \tilde p_u(t)\equiv 0$, or equivalenlty, 
\eqq{
-\md L_i^i \tilde{ y}_u(t) \equiv 0
}{nul}
The first row of the dynamics \eqref{tored} can now be rewritten in the new variables: 
$$ M_u \dot{\tilde y}_u(t) = -\md B_uW(\md B_d^T \tilde s_d(t) + \md B_u^T \tilde s_u(t))$$
By construction, $\dot{\tilde s}_u(t) = \tilde y_u(t)$. Using \eqref{sdd} and \eqref{tlu}, the dynamics of $\tilde s_u$ and $\tilde y_u$ satisfy 
\eqq{
\bp \dot{\tilde y}_u \\ \dot{\tilde s}_u \ep = \bp 0 & -M_u^{-1} \tilde{\md L}_u \\ I & 0 \ep \bp \tilde y_u \\ \tilde s_u \ep
}{dyna}
From \eqref{ru}, \eqref{nul} and \eqref{dyna}, it follows immediately that $col(\tilde y_u(t),  \tilde s_u(t) ) \in \ker(Obs(\hat C, \hat A))$ for all $t \geq 0$, which holds in particular for $t=0$. Finally, by combining \eqref{qq} and \eqref{sdd} for $t=0$, we see that $\tilde q(0)=(\md B_u^T-\md B_d^T(\md L_d^d+\md L_d^i)^{-1} \md L_i^i) \tilde s_u(0)$. We conclude that
\eqq{
\bp \tilde p_d(0) \\ \tilde p_u(0) \\ \tilde q(0) \ep  \in \hat Q \ker(Obs(\hat C, \hat A))
}{itis}
$(\spq)$ Take $col(\tilde y_u(0), \; \tilde s_u(0) ) \in \ker(Obs(\hat C, \hat A))$ and consider the trajectory $col(\tilde y_u(t), \; \tilde s_u(t) ) = e^{\hat A t} col(\tilde y_u(0), \tilde s_u(0) )$. We show that the trajectory $\tilde z(t)=$
\eqq{
\bp \tilde p_d(t) \\ \tilde p_u(t) \\ \tilde q(t) \ep  := \ub{ \bp 0 & 0 \\ M_u & 0 \\ 0 & \md \md B_u^T-\md B_d^T(\md L_d+\md K_d)^{-1}\md L_i \ep}{\hat Q}  \bp \tilde y_u(t) \\  \tilde s_u(t) \ep
}{newtraj}
is included in $\tilde{\mc B}^{LS}$. As $\tilde p_d(t)=0$ and $\tilde q(t) \in \im(\md B^T)$ for all $t \geq 0$, it remains to show that $\tilde z(t) = A \tilde z(t)$ for all $t \geq 0$.  Note that $\md B_dW \tilde q(t) =  \md L_i^i \tilde s_u(t)  - \md L_i^i \tilde s_u(t)=0$. Therefore, 
\eqq{
\dot{\tilde p}_d(t) = 0 = \ub{- R_dM_d^{-1} \tilde p_d(t)}{=0}  - \ub{\md B_dW \tilde q(t)}{=0}
}{x1}
By assumption, $R_uM_u^{-1} \tilde p_u \equiv 0$. Furthermore, from $\tilde q(t)=(\md B_u^T-\md B_d^T(\md L_d^d+\md L_d^i)^{-1}\md L_i^i)\tilde s_u(t)$, it follows that  $\md B_uW \tilde q(t)= \tilde{\md L}_u \tilde s_u(t)$. Also, $\dot{\tilde y}_u(t) = - \tilde M_u^{-1} \tilde{\md L}_u \tilde s_u(t)$, hence: 
\eqq{
\dot{\tilde p}_u(t) = M_u \dot{\tilde y}_u(t) = \ub{-R_uM_u^{-1} \tilde p_u(t)}{=0}- \md B_uW \tilde q(t)
}{x2}
By assumption, $\md L_i^i \tilde y_u(t) = \md L_i^i \dot{\tilde s}_u(t) \equiv 0$. As a consequence, 
$\dot{\tilde q}(t) = \md B_u^T \dot{\tilde s}_u(t) = \md B_u^T \tilde y_u(t) = \md B_u^T M_u^{-1} \tilde p_u(t)$ and therefore
\eqq{
\dot{\tilde q}(t) = \ub{\md B_d^TM_d^{-1}\tilde p_d(t)}{=0} + \md B_u^T M_u^{-1} \tilde p_u(t)
}{x3}
Taking together \eqref{x1}, \eqref{x2}, \eqref{x3}, we have 
\eq{
 \ub{\bp \dot{\tilde p}_d \\ \dot{\tilde p}_u \\ \dot{\tilde q} \ep}{\dot{\tilde z}} &= \ub{\bp -R_d M_d^{-1} & 0 & -\md B_dW 
\\ 0 & -R_uM_u^{-1} & -\md B_uW \\ \md B_d^TM_d^{-1} & \md B_u^TM_u^{-1} & 0  \ep}{A} \ub{\bp {\tilde p}_d \\ {\tilde p}_u \\ {\tilde q} \ep}{{\tilde z}} 
}
\end{proof}

$\mc S^{LS}$ can also be written as the unobservable subspace of the reduced system that gives the dynamics of the undamped nodes written in the $(\tilde p, \tilde q)$ coordinates: $\mc S^{LS} = \breve Q \ker(Obs(\breve C, \breve A))$, with
\eq{
\breve Q= \bp 0 & 0 \\ I & 0 \\ 0 & I \ep,
}
\eq{
\breve C =  \bp   0 & \md B_d W \\ R_uM_u^{-1} & 0 \\ 0& \md C^T  \ep, \; \;
\breve A = \bp 0 & -\md B_uW \\ \md B_u^TM_u^{-1} & 0 \ep
}
Here, the row and column decomposition of $\breve C$ and $\breve A$ are such that the first block is in accordance with $\tilde p_u$ and the second with $\tilde q$.  

{\bf Remark.} 
\bi
\item The state trajectories $col(\tilde y_u, \tilde s_u)$ of the system $(\hat C, \hat A)$ describe the behavior of the undamped nodes in the reduced graph with total Laplacian matrix $\tilde{\md L}_u$, which is obtained by eliminating the damped nodes according to a Kron reduction. This changes the topology including edge weights, but connectivity is preserved, hence $\ker(\tilde{\md L}_u)=\im( \I_r )$. $\hat Q$ represents the transformation matrix of the $(\tilde y_u(t), \tilde s_u(t))$ coordinates to the $(\tilde p_d(t), \tilde p_u(t), \tilde q_u(t))$ coordinates at steady state. Furthermore, $\md L^i_i \tilde y_u=0$ is an algebraic constraint that boils down to $\md B_d W \tilde q(t)=0$, i.e. zero net force at damped nodes in the original graph. Finally, the constraint $R_u \tilde y_u=0$ assures that partially undamped nodes can only move in directions in which they do not experience resistance. 
\item The system $(\breve C, \breve A)$ gives the dynamics of the undamped nodes when the damped nodes would be fixed at a single position. The set of solutions of $(\breve C, \breve A)$ for which $\md B_dW \tilde q(t)\equiv 0$, $R_uM_u^{-1} \tilde p_u(t) \equiv 0$ and $\tilde q(t) \in \im(\md B^T)$ for all $t \geq 0$ is equal to the set of long-run trajectories of $\tilde z$ of the system \eqref{y2} defined on $\Omega$ in which every solution is premultiplied by $\breve Q$.  
\ei

We show that all solutions in $\mc S^{LS}$ are composed of periodic functions, which are in fact sinusoids:

\begin{proposition}\label{nwg}  Each solution $\tilde z(t) \in \Omega$ of system \eqref{y2} can be written as a finite sum of sinusoids. 
\end{proposition}

\begin{proof} Note that $\breve A$ can be written as a product of a skew-symmetric matrix and a positive-definite diagonal matrix. Indeed, we have 

\[ \breve A= \underbrace{\bp 0 & -\md B_u \\ \md B_u^T  & 0 \ep}_{\breve B} \underbrace{\bp M_u^{-1} & 0 \\ 0 & W \ep}_{\breve W} \] 

Since $\breve A= \breve B \breve W$ is similar to the real skew-symmetric matrix $\breve W^{\frac{1}{2}} \breve B \breve W^{\frac{1}{2}}$, it has purely imaginary eigenvalues that are semisimple\footnote{A real skew-symmetric matrix is a normal matrix, which has the property to be diagonalizable.}. Take any solution $\tilde z(\cdot) \in \Omega$ of \eqref{y2}, then $\tilde p_d(t) \equiv 0$ and the trajectories $col(\tilde p_u(t), \tilde q(t) ) = e^{\breve{A} t} col(\tilde p_u(0),  \tilde q(0) )$ are composed of periodic functions of the form $\cos( \mathfrak I(\lambda_i) t)  \mathfrak R(x_i)+\sin( \mathfrak I(\lambda_i) t) \mathfrak I(x_i)$, where $x_i$ is an eigenvector of $\hat A$ and $\lambda_i$ is a semisimple, purely imaginary eigenvalue. \end{proof}

Each component $\cos( \mathfrak I(\lambda_i) t)  \mathfrak R(x_i)+\sin( \mathfrak I(\lambda_i) t) \mathfrak I(x_i)$ of $\tilde z(t)$ corresponds to a group of nodes that is oscillating with the same frequency $\mathfrak I(\lambda_i)$.  The nodes in this group are indicated by the non-zero entries in $x_i$. If undamped nodes belong to multiple oscillating groups, they might oscillate with multiple frequencies. 

Due to this periodic character of the components of the solutions, we cannot find a proper subset of $\mc S^{LS}$ to which all solutions converge. Hence,

\begin{corollary}\label{smallest} The smallest set to which all solutions of \eqref{y2} in $\Omega$ converge is given by $\mc S^{LS}$. 
\end{corollary}

From Lemma \ref{stable} and the periodic character of $\tilde z(t)$ at steady state, the solutions of \eqref{y2} in $\Omega$ are bounded. This has an important implication: the sum of the momenta of the undamped nodes turn out to be zero: 

\begin{proposition}\label{conservation} {(Conservation of momentum at steady state)} For any solution $\tilde z(t) = col( \tilde p_d(t), \tilde p_u(t), \tilde q(t) )$ of \eqref{y2} in $\Omega$, it holds that $\I_r^T \tilde p_u(t) \equiv 0$. 
\end{proposition}

In the $(\tilde y_u, \tilde s_u)$ coordinates, we find that conservation of momentum leads to $\I_r^T M_u \tilde y_u(t) \equiv 0$. Thus, $\I_r^T M_u \tilde y_u(t)$ might serve as an additional output variable to the system $(\hat C, \hat A)$ that does not affect the unobservable subspace. What is more, the same holds for its integral $\I_r^T M_u \tilde s_u(t)$ so that $\mc S^{LS}$ can be written equivalently as follows:

\begin{corollary}\label{bestbelangrijk2} $\mc S^{LS}= \hat Q \ker(Obs(\doublehat{C}, \hat A))$, where
\eq{
\doublehat{C} \isen  \bp  \md L_i^i & 0 \\ R_u & 0 \\ 0 & \I_r^T M_u \ep
}
\end{corollary}

\begin{proof} $(\sbq)$ Due to the freedom to choose an $\tilde s(0)$ that satisfies \eqref{qq}, we can choose one that satisfies $\I_r^T M_u \tilde s_u(0)=0$. To see that this is possible, consider solutions $\tilde s(t)=col( \tilde s_d(t),  \tilde s_u(t) )$ and $\tilde y_u(t)$ that satisfy \eqref{ru}, \eqref{qq}, \eqref{sdd}, \eqref{nul} and \eqref{dyna}. Replacing $\tilde s(t)$ by $\tilde s^*(t) = \tilde s(t) - \I_r \alpha$ with 
$$ \alpha = (\I_r^TM_u \I_r)^{-1} \I_r^TM_u \tilde s_u(0)$$ 
preserves these identities and furthermore $\I_r^T M_u \tilde s_u^*(0)=0$. Since $\I_r^T M_u \dot{\tilde s}^*_u(t)= \I_r^T M_u \dot{\tilde s}_u(t)=\I_r^T \tilde p_u(t) \equiv 0$, it holds that $\I_r^T M_u \tilde s_u^*(t) \equiv 0$. 
$(\spq)$ This follows from Lemma \ref{bestbelangrijk}: $\hat Q \ker(Obs(\doublehat{C}, \hat A)) \sbq \hat Q \ker(Obs({\hat C}, \hat A)) = {\mc S}^{LS}$
\qed\end{proof}

\subsection{Conditions on output consensus}

By combining Lemmas \ref{qook}, \ref{pede} and Corollaries \ref{smallest},  \ref{bestbelangrijk2}, we find that the pair $(\doublehat C,\hat A)$ is observable if and only output consensus is guaranteed:

\begin{proposition}\label{prop}
All output trajectories $ y(t)$ of $ \Sigma \times  \Gamma$ converges to a point in $\im\{ \I_r\}$ if and only if $\ker(Obs(\doublehat{C}, \hat A))=\{ 0 \}$.
\end{proposition}

\begin{proof} From Lemma \ref{qook}, all output trajectories $ y(t)$ of $ \Sigma \times  \Gamma$ converge to a point in $\im\{ \I_r\}$ if and only if every trajectory $\tilde z(t) = col ( \tilde p(t), \tilde q(t))$ of \eqref{y2} defined on $\Omega$ converges to the origin. Since $\mc S^{LS}$ is the smallest set to which all state trajectories of \eqref{y2} defined on $\Omega$ converge, that is equivalent to $\mc S^{LS}=\hat Q \ker(Obs(\doublehat C, \hat A))=\{ 0 \}$. It follows immediately that $\ker(Obs(\doublehat C, \hat A))=\{ 0 \}$ results in $\mc S^{LS}=\{0\}$. Now suppose that $\mc S^{LS}=\hat Q \ker(Obs(\doublehat C, \hat A))=\{0\}$, i.e. $\ker(Obs(\doublehat C,\hat A)) \sbq \ker(\hat Q)$. Note that $\ker(Obs(\doublehat C, \hat A)) \sbq \ker \bp 0 & \I_r^T M_u \ep$ and 
\eq{
\ker(\hat Q) \sbq & \ker  \haakk{ \bp  0 & M_u^{-1} & 0 \\ 0 & 0 & \md B_uW  \ep \hat Q } \issa \ker \bp I & 0 \\ 0 & \tilde{\md L}_u \ep = \im \bp 0 \\ \I_r \ep
}
Hence, 
\eq{
 & \ker(\hat Q) \cap \ker(Obs(\doublehat{C}, \hat A))  \\ \sbq & 
 \;  \im \bp 0 \\ \I_r \ep \cap \ker \bp 0 & \I_r^T M_u \ep=\{0\}
}
Then, by assumption we obtain $\ker(Obs(\doublehat{C}, \hat A))=\{0\}$. \end{proof}

We come to the following equivalence relation that connects the output consensus problem with the eigenspaces of $M_u^{-1}\tilde{\md L}_u$ and $M^{-1}\md L$:

\vspace{3mm}

\begin{theorem}\label{t1} The following is equivalent:

\begin{enumerate}[(i)]
\item Every plant output trajectory $y(t)$ of $\Sigma \times \Gamma$ converges to a point in the set $\im\{\I_r\}$.
\item None of the eigenvectors of $M_u^{-1} \tilde{\md L}_u$ is contained in the intersection of the kernel of $\md{L}_i^i$ and the kernel of $R_u$, i.e. for each $\mu \in \sigma(M_u^{-1} \tilde{\md L}_u)$:
\eqq{\ker(M_u^{-1} \tilde{\md L}_u- \mu I) \cap \ker \bp \md L_i^i \\ R_u \ep = \{ 0\}}{ii}
\item Every eigenvector of $M^{-1} \md L$ in the kernel of $R$ has at least one nonzero value in an entry that corresponds to a damped node, i.e. for each $\mu \in \sigma(M^{-1}\md L)$: 
\eqq{\ker(M^{-1}\md L- \mu I) \cap  \ker(R) \cap \im \bp 0 \\ I_{n_ur} \ep =\{0\}}{iii}
\end{enumerate}
\end{theorem}

\begin{proof} \textit{(i) $\iff$ (ii)} From Proposition \ref{prop}, condition (i) holds if and only if $\ker(Obs(\doublehat{C}, \hat A))=\{ 0 \}$. 
 According to Hautus lemma, that is equivalent to 

\[ \x{rank} \bp \hat A - \lambda I \\ \doublehat{C} \ep = 2n_u \qquad  \forall \lambda \in \mathbb{C} \]

Equivalently, from the rank-nullity theorem: $\forall \lambda \in \mathbb{C}$ it must hold that if

\begin{equation}\label{equagon} 
\bp -\lambda I & -M_u^{-1} \tilde{\md L}_u \\ I & - \lambda I\\  \md L_i^i & 0 \\ R_u & 0 \\ 0 & \I_r^TM_u \ep 
\bp  \tilde y_u \\ {\tilde{s}}_u \ep = 0 
\end{equation}

then $\tilde y_u=0$ and $\tilde{s}_u =0$. For $\lambda=0$, this implication always holds, since from the second block row it follows that ${\tilde y}_u=0$ and from the first block row, $\tilde s_u \in \im\{ \I_r \}$, which, combined with the bottom block row $\I_r^T M_u \tilde s_u=0$, yields $\tilde s_u=0$ (notice that $ \im( \I_r ) \cap \ker( \I_r^T M_u) =\{0\}$). 

So it remains to consider $\lambda \in \C \backslash \{0\}$, for which ${\tilde  y}_u= \lambda \tilde s_u$. Inserting this in the first block row yields $\lambda^2 {\tilde s}_u = - M_u^{-1} \tilde{\md L}_u  {\tilde s}_u$. Premultiplying by $\fone{\lambda^2} \md I_r^T M_u$ yields $\I_r^TM_u {\tilde s}_u =  -\frac{1}{\lambda^2} \I_r^T M_u M_u^{-1} \tilde{\md L}_u {\tilde y}_u=0$, hence the last block row is always satisfied if the block rows above are satisfied too. Thus, for all $\lambda \neq 0$, the only solutions of \eqref{equagon} are $\tilde y_u=0$, $\tilde s_u=0$ if and only if for all $\lambda \neq 0$,
\eqq{
 M_u^{-1}  \tilde{\md L}_u { \tilde s}_u  \isen \; -\lambda^2 {\tilde s}_u \\ 
\lambda \md L_i^i {\tilde s}_u \isen \; 0 \\
\lambda R_u \tilde s_u \isen \; 0
}{kkkw}

implies ${\tilde s}_u=0$ (and therefore also $\tilde y_u={\lambda} \tilde s_u=0$). That is, for any eigenvalue $\mu=-\lambda \in \sigma(M_u^{-1}\tilde{\md L}_u) \backslash \{0\}$, \eqref{ii} holds. Since for $\mu=0 \in \sigma(M_u^{-1}\tilde{\md L}_u)$, $\ker(M_u^{-1} \tilde{\md L}_u - \mu I) = \im(\I_r)$ and $\im(\I_r) \cap \ker(\md L_i^i)=\{0\}$\footnote{This holds since $\md L_i^i$ is a nonzero and nonpositive matrix}, \eqref{ii} always holds for $\mu=0$.

\textit{(ii) $\iff$ (iii)} Write out $\tilde{\md L}_u$ in the left-hand side of the first equation in \eqref{kkkw} where $\lambda \neq 0$ and use the second constraint to obtain: 
\eq{
M_u^{-1}  \tilde{\md L}_u {\tilde s}_u  &= M_u^{-1} (\md L_u^u+\md L_u^i -  (\md L_i^i)^T(\md L_d^d+\md L_d^i)^{-1}\md L_i^i) {\tilde s}_u \\
&=  M_u^{-1}(\md L_u^u+\md L_u^i) {\tilde s}_u 
}
With this and the fact that $\ker(\md L_i^i)=\ker(M_d^{-1}\md L_i^i)$, the first two identities in \eqref{kkkw} are equal to 
\eq{ 
\bp M_d^{-1} & 0 \\ 0 & M_u^{-1} \ep \bp \md L_d^d+\md L_d^i & \md L_i^i \\  (\md L_i^i)^T & \md L_u^u+ \md L_u^i \ep \bp 0 \\ {\tilde s}_u \ep &=  - \lambda^2 \bp 0 \\ \tilde s_u \ep
}
The last identity in \eqref{kkkw} can be rewritten as
\eq{
\bp R_d & 0 \\ 0 & R_u \ep \bp 0 \\ \tilde s_u \ep = \bp 0 \\ 0 \ep
}
Hence, \eqref{ii} is true for any $\mu \in \sigma(M_u^{-1} \tilde{\md L}_u) \backslash \{0\}$  if and only if there does not exist an eigenvector of $M^{-1}\md L$ corresponding to a nonzero eigenvalue that is in the kernel of $R$ and of the form $col( 0, \; {\tilde s}_u)$. Also, \eqref{iii} is always true for $\mu=0$. Thus, the latter condition is equivalent to (iii).
\end{proof}


We give the following corollary without proof:  

\begin{corollary}
By a change of coordinates, we can extend Theorem \ref{t1} with the following equivalent conditions: 
\begin{enumerate}[(i)]
\setcounter{enumi}{3}
\item None of the eigenvectors of $\tilde{\md L}_u M_u^{-1}$ is contained in the intersection of the kernels of $\md L_iM_u^{-1}$ and $R_u M_u^{-1}$. 
\item Every eigenvector of $\md L M^{-1}$ in the kernel of $RM^{-1}$ has at least one nonzero value in an entry that corresponds to a damped node.
\end{enumerate}
\end{corollary}

\section{Conclusion}

We considered a class of qualitatively heterogeneous networked models that includes mass-spring-damper networks and studied the output consensus problem of determining whether the plant output trajectories converge to an agreement value. This problem can be tackled by performing a stability analysis of a shifted model of the network (Lemma \ref{qook}). If this system is not globally asymptotically stable (GAS), then some of the undamped nodes exhibit oscillatory behavior at steady state. The oscillation space can be obtained from the unobservable subspace of a reduced system that gives the dynamics of the undamped nodes at steady state (Lemma \ref{bestbelangrijk}). In a steady state, the nodes show conservation of momentum (Lemma \ref{conservation}). This brings us at a system that is observable if and only if output consensus is guaranteed (Lemma \ref{prop}). Alternatively, the consensus problem is equivalent to an eigenspace problem that  depends on the graph topology, the edge weights, the mass values and the resistance values of the undamped nodes, see Theorem \eqref{t1}. Since the results show that the topology plays a major role in determining GAS, an obvious topic for future research is to find sufficient conditions purely based on the topology of the graph. Such results can be helpful in e.g. the pinning control problem, where one is looking for a strategy to place a minimal number of damped nodes in order to ensure output consensus.


\begin{thebibliography}{99}

\bibitem{c2} M. B{\"u}rger, D. Zelazo and F. Allg{\"o}wer, "Hierarchical Clustering of Dynamical Networks Using a Saddle-Point Analysis", IEEE Transactions on Automatic Control, vol. 58, issue 1, pp. 113-124, 2013. 
\bibitem{c1} M. B{\"u}rger, D. Zelazo and F. Allg{\"o}wer, "Duality and Network Theory in Passivity-based Cooperative Control", Automatica, vol. 50, pp. 2051-2061, August 2014. 
\bibitem{seyboth} G.S. Seyboth, D.V. Dimarogonas, K.H. Johansson, P. Frasca, "On Robust Synchronization of Heterogeneous Linear Multi-Agent Systems with Static Couplings", Automatica, vol. 53, pp. 392-399, March 2015. 
\bibitem{grip} H.F. Grip, T. Yang, A. Saberi, A.A. Stoorvogel, "Output Synchronization for Heterogeneous Networks of Non-Introspective, Non-Right-Invertable Agents", IEEE American Contr. Conf., pp. 5791-5796, 2013.
\bibitem{olfati} R. Olfati-Saber, A. Fax, R.M. Murray, "Consensus and Cooperation in Networked Multi-Agent Systems", Proceedings of the IEEE, vol. 95, issue 1, pp. 215-233, 2007. 
\bibitem{zhao} J. Zhao, D.J. Hill, T. Liu, "Synchronization of Dynamical Networks with Nonidentical Nodes: Criteria and Control", IEEE Transactions on Circuits and Systems I: Regular Papers, vol. 58, issue 3, pp. 584-594, 2011.
\bibitem{arcak} M. Arcak, "Passivity as a Design Tool for Group Coordination", IEEE Transactions Automatic Control, vol. 52, issue 8, pp. 1380-1390, 2007. 
\bibitem{murguia} C. Murguia, R.H.B. Fey, H. Nijmeijer, "Partial Network Synchronization and Diffusive Dynamic Couplings",  Proceedings of the 19th IFAC World Congress, vol. 19, part 1, 2014.
\bibitem{bergenhill} D. Hill, A. Bergen, "Stability Analysis of Multimachine Power Networks with Linear Frequency Dependent Loads", IEEE Transactions on Circuit and Systems, vol. 29, issue 12, pp. 840-848, March 1982.
\bibitem{traffic}D. Helbing, "Traffic and related self-driven many-particle systems", Rev. Mod. Phys., vol. 73, pp. 1067-1141, December 2001
\bibitem{port} A. van der Schaft and D. Jeltsema, 
"Port-Hamiltonian Systems Theory: An Introductory Overview", 
Foundations and Trends in Systems and Control,
vol. 1, no. 2/3, pp. 173-378, 2014. 
\bibitem{kron} F. D{\"o}rfler, F. Bullo,
"Kron reduction of graphs with applications to electrical networks",
IEEE Transactions on Circuits and Systems I: Regular Papers, vol. 60, issue 1, pp. 150 - 163, 2013.
\bibitem{graphmatrices} S. Pirzada, "An introduction to graph theory", University Press Hyderabad, 2012. 
\bibitem{observability} G. Parlangeli and G. Notarstefano, "On the Reachability and Observability of Path and Cycle Graphs", IEEE Transactions on Automatic Control, vol. 57, pp. 743-748, 2012. 
\bibitem{polderman} J. W. Polderman and J. C. Willems, "Introduction to the Mathematical Theory of System
and Control: A Behavioral Approach", ISBN 0-387-98266-3, Springer-Verlag, New York, 1998.
\bibitem{port2} A. J. van der Schaft and B. M. Maschke, "Port-Hamiltonian Systems on Graphs", SIAM J. Control Optim., vol. 51(2), pp. 906-937.
\bibitem{rahmani} A. Rahmani, M. Ji, M. Mesbahi and M. Egerstedt, "Controllability of Multi-Agent Systems from a Graph-Theoretic Perspective", SIAM J. Control Optim., vol. 48(1), pp. 162-186.
\bibitem{chen} G. Chen, "Pinning Control and Synchronization on Complex Dynamical Networks", International Journal of Control, Automation, and Systems, vol. 12(2), pp. 221-230, 2014.
\bibitem{trip} S. Trip, M. B{\"u}rger, and C. De Persis, "An Internal Model Approach to Frequency Regulation in Inverter-based
Microgrids With Time-varying Voltages", IEEE Conf. on Dec. and Control, pp. 223-228, 2014.\end{thebibliography}
\end{document}